\documentclass[12pt]{article} 

\usepackage[T2A]{fontenc} 
\usepackage[utf8]{inputenc} 
\usepackage{comment, amsthm, amssymb, amsmath, mathrsfs, bbm} 
\usepackage[english]{babel}
\newtheorem{theorem}{Theorem} 
\newtheorem{lemma}{Lemma}[section]
\newtheorem{cor}{Corollary}[section]
\theoremstyle{remark}
\theoremstyle{Example}

\theoremstyle{Fact}

\theoremstyle{Problem}

\theoremstyle{Solution}

\theoremstyle{Corollary}

\theoremstyle{Proposition}

\theoremstyle{Sketch of the proof}

\newtheorem{rem}{Remark}[theorem]

\theoremstyle{definition}
\theoremstyle{plain}
\newtheorem{definition}{Definition}[section]

\begin{document}

\author{Dmitry Krachun}
\title{Extreme Gibbs measures for high-density hard-core model on $\mathbb{Z}^2$. }
\date{}
\maketitle 

\newcommand{\eps}{\varepsilon}

\begin{abstract}
We confirm the list from \cite{MSS} of values $D$ for which the high-density hard-core model on $\mathbb{Z}^2$ with exceptional distance $D$ has infinitely many extremal Gibbs states. As a byproduct, we prove that for all $D>0$ there exists an acute-angled triangle inscribes in $\mathbb{Z}^2$ with side-lengths at least $D$ and area $\sqrt{3}/4\cdot D^2+O(D^{4/5})$. 
\end{abstract}

\section{Introduction}

In the paper \cite{MSS}, authors study extremal Gibbs measures for the high-density hard-core model on the square lattice. It turns out that the structure of extremal Gibbs measures is very sensitive to arithmetic properties of the hard-core exclusion distance $D$, in relation to the set $\mathcal{T}(D)$ of acute-angled triangles inscribed in $\mathbb{Z}^2$ with all sides at least $D$. Depending on the number and form of triangles from $\mathcal{T}(D)$ of smallest area, possible values of $D$ are split into three classes leading to different structures of Gibbs states. We refer to \cite{MSS} for details on the model and the study of its Gibbs measures.

One of the classes is formed by values of $D$ that generate \emph{sliding}. We say that sliding occurs if there exist two triangles from the set $\mathcal{S}(D)$ which share two vertices, see Definition \ref{def:sliding} for a formal definition. In the presence of sliding, there are countably many periodic ground states. It was conjectured in \cite{MSS} that sliding occurs only for finitely many values of $D$. Moreover, a conjectured list of values $D$ for which sliding occurs was given there. In this paper we prove that the list from \cite{MSS} is indeed complete.

\begin{theorem}\label{thm:main}
The list of \cite{MSS} is the only set of values with sliding.
\end{theorem}

Our proof is number-theoretic and goes via estimating the smallest possible area of a triangle in $\mathcal{T}(D)$. As a byproduct we prove that there exists a triangle from $\mathcal{T}(D)$ with area at most $\sqrt{3}/2\cdot D^2+O(D^{4/5})$. We also present a short geometric argument showing that if sliding occurs then this area cannot me smaller than $\sqrt{3}/2\cdot D^2+\frac{D}{2\sqrt{3}}-1$.

The paper is organised as follows. In Section \ref{sec:upper-bound} we show an upper bound on the smallest possible area of a triangle from $\mathcal{T}(D)$. Then, in Section \ref{sec:explicit}, using a similar technique, we prove an explicit but weaker bound. Finally, in Section \ref{sec:sliding} we show an upper bound for this area when sliding occurs, thus confirming a list of values $D$ with sliding from \cite{MSS}.

\section{Smallest possible area.}\label{sec:upper-bound}

The aim of this section is to show a general upper bound on the smallest possible area of an acute-angled triangles inscribed in $\mathbb{Z}^2$ with side-lengths at least $D$. We start with some definitions.

\begin{definition}
For a positive $D$ we define $\mathcal{T}(D)$ to be the set of all triangles with vertices in $\mathbb{Z}^2$, side-lengths at least $D$, and angles at most $\pi/2$. Following \cite{MSS}, we denote twice the smallest area of a triangle from $\mathcal{T}(D)$ by $S(D)$. We refer to triangles from $\mathcal{T}(D)$ of minimal area as M-triangles.
\end{definition}

\begin{definition}\label{def:sliding}
We say that for $D>0$ \emph{sliding} occurs if there exists two M-triangles $OWA$ and $OWB$ with $A$ and $B$ being on the same side of $OW$. We refer to $OW$ as to the \emph{base} of the triangle.
\end{definition}

In the following, for a real $x$ we denote by $\{x\}$ the fractional part of $x$ and by $\|x\|$ the distance from $x$ to the closest integer. We also use the notation $O(f(n))$ to denote any function which is upper-bounded by $C\cdot f(n)$ for some absolute constant $C>0$. We start from a reduction to a number theoretic problem.

\begin{lemma}\label{lm:reduction}
Let $D\geq 10$ and $\eps<1/2$ and suppose there exist integers $x$ and $y$ such that 
\begin{equation}\label{eq:conditions}
\text{(1)}\|\sqrt{3}x\|<\eps\qquad 
\text{(2)}\|\sqrt{3}y\|<\eps\qquad 
\text{(3)}2\sqrt{x^2+y^2}\in [D+\sqrt{2}\eps, D+2\eps ],
\end{equation}
Then $S(D)<\sqrt{3}/2\cdot D^2+(2\sqrt{3}+\sqrt{2})D\eps+7\eps^2$. 
\end{lemma}
\begin{proof}
Note that points $A(0,0), B(2x, 2y)$ and $C(x-\sqrt{3}y, y+\sqrt{3}x)$ form an equilateral triangle with side length between $D+\sqrt{2}\eps$ and $D+2\eps$. Let $z$ and $t$ be integers closest to $x-\sqrt{3}y$ and $y+\sqrt{3}x$ respectively. First two conditions ensure that $|z-(x-\sqrt{3}y)|<\eps$ and $|t-(y+\sqrt{3}x)|<\eps$, and so points $C'(z,t)$ and $C(x-\sqrt{3}y, y+\sqrt{3}x)$ are at most $\sqrt{2}\eps$ apart. Thus, by triangle inequality, all sides of the triangle $ABC'$ are at least $D$ and at most $D+1$. Hence, $ABC'$ is acute-angled, and so $S(D)$ is at most twice its area. We then have
\begin{align*}
S(D)&\leq 2\cdot |xt-yz| \leq 2x(y+\sqrt{3}x+\eps)-2y(x-\sqrt{3}y-\eps)
\\&=
2\sqrt{3}(x^2+y^2)+2\eps(x+y)\leq \sqrt{3}/2\cdot (D+2\eps)^2 + \sqrt{2}\eps(D+2\eps)
\\&\leq
\sqrt{3}/2\cdot D^2+ (2\sqrt{3}+\sqrt{2})D+7\eps^2. 
\end{align*}
\end{proof}

\begin{rem}
Alternatively, to upper-bound the area of $\mathfrak{T}$ one can note that all sides of $\mathfrak{T}$ are at most $D+(2+\sqrt{2})\eps$.
\end{rem}

In order to find a pair $(x,y)\in\mathbb{Z}^2$ satisfying (\ref{eq:conditions}) we need a simple lemma  which is probably standard.

\begin{lemma}\label{lm:approx}
There exists an absolute constant $C>0$ such that for every $\eps<1$ and $x\in[0, 1]$ there exists an integer $s\leq C/\eps$ such that $\{s\sqrt{3}\}\in [x, x+\eps]$. 
\end{lemma}
\begin{proof}
It suffices to prove the statement for $\eps_k:=10^{-k}$. We use induction on $k$. The base is trivial. Suppose we want to prove the statement for $k+1$ and some $x$. Using induction hypothesis we can find $s\leq C/\eps_k$ such that $\{s\sqrt{3}\}\in [x-\eps_k, x]$. Then choose $v\leq 10/\eps_{k+1}$ such that $\{v\sqrt{3}\}<\eps_{k+1}$ (for example, by solving Pell's equation $3v^2-2=u^2$). It is easy to see that we always have $\{v\sqrt{3}\}>1/(3v)$ and so by adding $v$ to $s$ at most $\eps_k\cdot 3v \leq 300$ times we get $s'$ with $\{s'\sqrt{3}\}\in [x, x+\eps_{k+1}]$. We have $s'\leq 300v+s\leq 3000/\eps_{k+1}+C/\eps_k< C/\eps_{k+1}$, provided $C$ is large enough.
\end{proof}
\begin{cor}\label{cor:closeint}
There exists an absolute constant $C>0$ such that for every $\eps<1$ and every $n\in\mathbb{R}$ there exists an integer $k\in [n-C/\eps, n]$ such that $\|k\sqrt{3}\|<\eps$.
\end{cor}
\begin{proof}
By increasing $C$ slightly, we may assume that $n$ is an integer. Let $x:=\{n\sqrt{3}\}$. By Lemma \ref{lm:approx} there exists an integer $s\leq C/\eps$ such that $\{s\sqrt{3}\}$ is $\eps$-close to $x$ and then $k:=n-s$ works.
\end{proof}

We now construct a pair $(x, y)$ satisfying conditions similar to (\ref{eq:conditions}).

\begin{lemma}\label{lm:findpairs}
There exists an absolute constant $C_1>0$ such that for any $N>100$ and $\eps<1$ there exists $(x,y)\in\mathbb{Z}^2$ satisfying 
\[
\text{(1)}\|\sqrt{3}x\|<\eps\qquad 
\text{(2)}\|\sqrt{3}y\|<\eps\qquad 
\text{(3)}x^2+y^2\in [N, N+C_1(\eps^{-2}+N^{1/4}\eps^{-3/2})],
\]
\end{lemma}
\begin{proof}
Let $C$ be the constant from Corollary \ref{cor:closeint}. By Corollary \ref{cor:closeint} choose $x\in(\sqrt{N}-C/\eps, \sqrt{N})$ such that $\|x\sqrt{3}\|<\eps$. Then choose $y\in (\sqrt{N-x^2}, \sqrt{N-x^2}+C/\eps)$ with $\|y\sqrt{3}\|<\eps$. We then have 
\[
0<N-x^2<2\cdot (C/\eps)\cdot \sqrt{N}=2C\cdot N^{1/2}\eps^{-1},
\]
and so
\[
0<x^2+y^2-N<(C/\eps)^2+2\cdot (C/\eps)\cdot \sqrt{N-x^2}<C^2 \eps^{-2}+2C^{3/2}N^{1/4}\eps^{-3/2},
\]
and we can choose $C_1=C^2$.
\end{proof}

Finally, we prove 
\begin{theorem}\label{thm:upper-bound}
Asymptotically, for large $D$, one has $S(D)=\sqrt{3}/2\cdot D^2+O(D^{4/5})$.
\end{theorem}
\begin{proof}
Take $\eps:=C\cdot D^{-1/5}$ with large enough constant $C>0$. Define $N:=\left(\frac{D+\sqrt{2}\eps}{2}\right)^2$. By Lemma \ref{lm:findpairs} we can find $(x, y)\in\mathbb{Z}^2$ with $\|\sqrt{3}x\|<\eps$, $\|\sqrt{3}y\|<\eps$ and $x^2+y^2\in [N, N+C_1(\eps^{-2}+N^{1/4}\eps^{-3/2})]$. We then use Lemma \ref{lm:reduction} to obtain 
\[
S(D)<\sqrt{3}/2\cdot D^2+ (2\sqrt{3}+\sqrt{2})D+7\eps^2=\sqrt{3}/2\cdot D^2+O(D^{4/5}). 
\]
 We now only need to check that the upper bound on $x^2+y^2$ given by Lemma \ref{lm:findpairs} matches the upper bound needed to use Lemma \ref{lm:reduction}.
\begin{align*}
2\sqrt{x^2+y^2}&<2\left(N+C_1(\eps^{-2}+N^{1/4}\eps^{-3/2})\right)^{1/2}=2\sqrt{N}+O(N^{-1/4}\eps^{-3/2})
\\&=D+\sqrt{2}\eps+O(D^{-1/2}\eps^{-3/2})<D+2\eps,
\end{align*}
for $\eps=C\cdot D^{-1/5}$ and an absolute constant $C$ large enough.
\end{proof}

\section{Explicit bounds on $S(D)$}\label{sec:explicit}
In this section we make quantitative the bounds from Section \ref{sec:upper-bound}. The resulting lower bound on $S(D)$ is asymptotically worse for large $D$; nevertheless, it still suffices to rule out the possibility of sliding.

\begin{lemma}\label{lm:approx-seven}
Let $\eps=(3\sqrt{3}-5)/2\approx .098$, and $v$ be an arbitrary real number. Suppose a segment $J\subset\mathbb{R}$ has length at least $7$. Then there exist an integer $x\in J$ and a half-integer $y\in J$ such that $\|\sqrt{3}x+v\|, \|\sqrt{3}y+v\|\leq \eps$.
\end{lemma}
\begin{proof}
Let $n$ be the largest integer in $J$. Then $n-s\in J$ for $s=0,1,\dots, 6$. On the unit circle, numbers $0, \{\sqrt{3}\}, \{2\sqrt{3}\}, \{3\sqrt{3}\}, \{4\sqrt{3}\}, \{5\sqrt{3}\}, \{6\sqrt{3}\}$ have maximal gap of $2\eps$, which implies that $\{n\sqrt{3}+v\}$ is $\eps$ close to at least one of them; and if $\{n\sqrt{3}+v\}$ is $\eps$ close to $\{s\sqrt{3}\}$ then $\{(n-s)\sqrt{3}+v\}$ is $\eps$ close to an integer. The proof for half-integers is the same but one needs to take $n$ to be the largest half-integer in $J$.
\end{proof}
\begin{lemma}\label{lm:upper-bound-explicit}
Let $\delta:=0.13$. Let $D>1000$ and $\eps=(3\sqrt{3}-5)/2\approx .098$ and suppose there exist $x, y\in \frac{1}{2}\cdot \mathbb{Z}$ such that 
\begin{equation}
\text{(1)}\|x-\sqrt{3}y\|<\eps\qquad 
\text{(2)}\|y+\sqrt{3}x\|<\eps/2\qquad 
\text{(3)}2\sqrt{x^2+y^2}\in [D+\sqrt{5}/2\cdot \eps, D+\delta],
\end{equation}
Then $S(D)<\sqrt{3}/2\cdot D^2+(\sqrt{3}\delta+\eps/2)\cdot D+ 2\eps y + 0.02$
\end{lemma}
\begin{proof}
Again, note that points $A(0,0), B(2x, 2y)$ and $C(x-\sqrt{3}y, y+\sqrt{3}x)$ form an equilateral triangle with side length between $D+\sqrt{5}/2\cdot \eps$ and $D+\delta$. Let $z$ and $t$ be integers closest to $x-\sqrt{3}y$ and $y+\sqrt{3}x$ respectively and consider $C'(z,t)$. First two conditions ensure that $|z-(x-\sqrt{3}y)|<\eps$ and $|t-(y+\sqrt{3}x)|<\eps/2$. Triangle inequality then implies that all sides of $ABC'$ are at least $D$. Then $S(D)$ is bounded by twice the area of the triangle $ABC'$ which is
\begin{align*}
2\cdot S(ABC')
&=
|2xt-2yz|\leq 2x(y+\sqrt{3}x+\eps/2)-2y(x-\sqrt{3}y-\eps)
\\&=
2\sqrt{3}\cdot (x^2+y^2)+\eps(x+2y) \leq 
2\sqrt{3}\cdot \left(\frac{D+\delta}{2}\right)^2+\eps(x+2y)
\\&=
\sqrt{3}/2\cdot D^2+(\sqrt{3}\delta+\eps/2)\cdot D+2\eps y + \eps(x-D/2)+\sqrt{3}\delta^2/2.
\end{align*}
And it remains to note that $x-D/2\leq \delta/2$ and $\delta\eps/2+\sqrt{3}\delta^2/2\leq 0.02$. 
\end{proof}
\begin{lemma}\label{lm:bound-N}
For every $N>1000$ there exist $x,y\in \frac{1}{2}\cdot \mathbb{Z}$ such that
\begin{equation}
\text{(1)}\|x-\sqrt{3}y\|<\eps\qquad 
\text{(2)}\|y+\sqrt{3}x\|<\eps/2\qquad 
\text{(3)}x^2+y^2\in [N, N+f(N)],
\end{equation}
where $f(N):=49+14\cdot\left(7\sqrt{N}-49/4\right)^{1/2}$. Moreover, we can take $y\leq 7+\left(7\sqrt{N}-49/4\right)^{1/2}$.
\end{lemma}
\begin{proof}
First, we can choose an integer $2x\in [2N^{1/2}-7, 2N^{1/2}]$ such that $\|\sqrt{3}\cdot 2x\|\leq \eps$. Then $\sqrt{3}x$ is $\eps/2$ close to an integer or a half integer. Lemma \ref{lm:approx-seven} then shows that for any segment $J$ of length at least $7$ we can choose $y\in \frac{1}{2}\cdot \mathbb{Z}$ such that first two conditions hold. The choice $J=[(N-x^2)^{1/2}, (N-x^2)^{1/2}+7]$ ensures that $x^2+y^2\geq N$ and 
\[
x^2+y^2-N^2\leq 14\cdot (N-x^2)^{1/2} + 49 \leq 14\cdot \left(N-(\sqrt{N}-7/2)^2\right)^{1/2}+49=f(N).
\]
We also have 
\[
y\leq (N-x^2)^{1/2}+7\leq \left(N-(N^{1/2}-7/2)^2\right)^{1/2}+7=7+\left(7\sqrt{N}-49/4\right)^{1/2}. 
\]
\end{proof}
We now combine Lemma \ref{lm:upper-bound-explicit} and Lemma \ref{lm:bound-N} to obtain
\begin{lemma}\label{lm:large-D}
For $D>7\cdot 10^6$ one has $S(D)< \sqrt{3}/2\cdot D^2+\frac{D}{2\sqrt{3}}-1$.
\end{lemma}
\begin{proof}
We take $\eps=(3\sqrt{3}-5)/2$ and $\delta:=0.13$ as in Lemma \ref{lm:upper-bound-explicit}. Let $N=\left(\frac{D+\sqrt{5}/2\cdot\eps}{2}\right)^2$. Note that $D > 6.7\cdot 10^6$ implies
\[
N+f(N)>\left(\frac{D+\delta}{2}\right)^2,
\] 
where $f(N)$ is as in Lemma \ref{lm:bound-N}. So we can apply Lemma \ref{lm:upper-bound-explicit} and Lemma \ref{lm:bound-N} to obtain
\[
S(D)-\sqrt{3}/2\cdot D^2\leq (\sqrt{3}\delta+\eps/2)\cdot D+2\eps\cdot \left(7+\left(7\sqrt{N}-49/4\right)^{1/2}\right)+0.02<\frac{D}{2\sqrt{3}}-1,
\]
where the last inequality is true for all $D>1590$. 
\end{proof}
\begin{lemma}\label{lm:small-D}
For $D\in [2\cdot10^3, 7\cdot 10^6]$ one has $S(D)< \sqrt{3}/2\cdot D^2+\frac{D}{2\sqrt{3}}-1$.
\end{lemma}
\begin{proof}
For $D\in [2.3\cdot 10^5, 6.7\cdot 10^6]$ we use Lemma \ref{lm:reduction} with $\eps=\frac{1}{17}$. By this lemma, if there exist integers $x, y$ such that
\begin{equation}
\text{(1)}\|\sqrt{3}x\|<\eps\qquad
\text{(2)}\|\sqrt{3}y\|<\eps\qquad
\text{(3)}2\sqrt{x^2+y^2}\in [D+\sqrt{2}\eps, D+2\eps ],
\end{equation}
then 
\[
S(D)<\sqrt{3}/2\cdot D^2+\frac{2\sqrt{3}+\sqrt{2}}{17}\cdot D + \frac{7}{289}< \sqrt{3}/2\cdot D^2+\frac{D}{2\sqrt{3}}-1,
\]
where the last inequality is true for all $D>600$. In order to check that there exist $x, y$ satisfying all three properties, we use a computer: We list all $x\leq 3.5\cdot 10^6$ satisfying $\|\sqrt{3}x\|<\eps$ and go through all pairs $x\leq y$ to compute $2\sqrt{x^2+y^2}$. We then check that any interval of length $(2-\sqrt{2})\eps\approx 0.034$ contains at least one such number. 

For numbers $D\in [2\cdot 10^3, 2.3\cdot 10^5]$ we simply list all possible triangles of interest: For every pair $(x, y)$ with $\|2\sqrt{3}x\|, \|2\sqrt{3}y\|\leq 1/10$ and $x^2+y^2<10^{12}$ we compute numbers $z$ and $t$ closest to $\frac{x-\sqrt{3}y}{2}$ and $\frac{y+\sqrt{3}x}{2}$ respectively and then add a triangle with vertices $(0,0), (x, y)$ and $(z, t)$ to a list. We then compute its smallest side and its area to compute (an upper bound for) $S(D)$.
\end{proof}

\section{Sliding}\label{sec:sliding}
In this section we give a lower bound for $S(D)$ in the presence of sliding.

\begin{theorem}\label{thm:sliding}
In the presence of sliding for $D\geq 1$ one has $S(D)>\sqrt{3}/2\cdot D^2+\frac{D}{2\sqrt{3}}-1$. 
\end{theorem}
\begin{proof}
Let $OW$ be the base and $A$ and $B$ be remaining vertices of two M-triangles, see Definition \ref{def:sliding}. Since areas of all M-triangles is $S(D)$, line $AB$ is parallel to the base $OW$. By switching $A$ and $B$ if necessary, we may assume that points $O, A, B, W$ form a trapezoid in this order.
It is an easy check that in any trapezoid one has $|OB|^2+|AW|^2=|AO|^2+|BW|^2+2|AB|\cdot|OW|$. Using the fact that all sides of maximal triangles are at least $D$ and $|AB|\geq 1$, we obtain
\[
|OB|^2+|AW|^2\geq 2D^2+2D.
\]
So at least one of $OB$ and $AW$ is at least $(D^2+D)^{1/2}$. Hence, $S(D)$ is at least twice the area of a triangle with sides $(D, D, (D^2+D)^{1/2})$ which is at least $\sqrt{3}/2\cdot D^2+\frac{D}{2\sqrt{3}}-1$ for $D\geq 1$.
\end{proof}

Lemma \ref{lm:large-D}, Lemma \ref{lm:small-D}, Theorem \ref{thm:sliding}, and computations from \cite{MSS} immediately prove Theorem \ref{thm:main}.

\vspace{1cm}
\hspace{-0.66cm}\textbf{Acknowledgement.} I would like to thank A. Mazel, I. Stuhl, and Y. Suhov for proofreading an earlier version of the note.

\end{document}